\documentclass[11pt]{article}
\usepackage{amssymb}
\usepackage{amsthm}
\usepackage{amsmath}
\usepackage{txfonts}
\usepackage{bbold}
\usepackage{mathtools}
\DeclarePairedDelimiter{\ceil}{\lceil}{\rceil}
\DeclarePairedDelimiter{\floor}{\lfloor}{\rfloor}

\usepackage{color}

\usepackage[colorlinks=true,citecolor=black,linkcolor=black,urlcolor=blue]{hyperref}

\theoremstyle{plain}
\newtheorem{THM}{Theorem}[section]
\newtheorem*{THM*}{Theorem}
\newtheorem{PROP}[THM]{Proposition}
\newtheorem{LEMMA}[THM]{Lemma}

\theoremstyle{definition}

\date{}
\theoremstyle{definition}

\theoremstyle{remark}

\newcommand{\pr}[1]{\mathbb{P} \left[ #1 \right]}

\newcommand{\arXiv}[1]{Preprint available on \url{http://arxiv.org/abs/#1}}

\newcommand{\card}[1]{\left| #1 \right|}

\newcommand{\eps}{\varepsilon}

\newcommand{\Kmn}{K_{m\otimes n}}
\def\mathyp{{\hbox{-}}}

\newcommand{\cheg}{u_1 \mathyp u_2 \mathyp u_3} % cherry in G
\newcommand{\chegp}{u_1^\prime \mathyp u_2^\prime \mathyp u_3^\prime} %cherry in G with primes
\newcommand{\chek}{[v_1 v_2 v_3]} % mono cherry in kmn
\newcommand{\chekp}{[v_1^\prime v_2^\prime v_3^\prime]} % mono cherry in kmn with primes
\newcommand{\baduv}{B_{\cheg}^{\chek}} % bad event
\newcommand{\baduvp}{B_{\chegp}^{\chekp}} % bad event wiht primes

\newcommand{\quag}{(u_1 u_2)(u_3 u_4)}
\newcommand{\quagp}{(u_1^\prime u_2^\prime)(u_3^\prime u_4^\prime)}
\newcommand{\quak}{[v_1 v_2 v_3 v_4]}
\newcommand{\quakp}{[v_1^\prime v_2^\prime v_3^\prime v_4^\prime]}
\newcommand{\caduv}{B_{\quag}^{\quak}}
\newcommand{\caduvp}{B_{\quagp}^{\quakp}}

\newcommand{\ef}{(e_1, e_2)}
\newcommand{\efp}{(e_1^\prime, e_2^\prime)}
\newcommand{\badet}{B^{\ef}_{\, \tau}}
\newcommand{\badetp}{B^{\efp}_{\,\tau ^\prime}}

\newcommand{\cnrl}{C_n^{(r)}(\ell)}
\newcommand{\Knr}{K_n^{(r)}}

\title{\vspace{-0.8cm} Bounded colorings of multipartite graphs and hypergraphs\thanks{Research supported by SNSF grant 200021-149111.}}
\author{
Nina Kam\v{c}ev \thanks{Department of Mathematics, ETH, 8092 Zurich. Email: nina.kamcev@math.ethz.ch.}
\and
Benny Sudakov \thanks{Department of Mathematics, ETH, 8092 Zurich. Email: benjamin.sudakov@math.ethz.ch.}
\and
Jan Volec \thanks{Department of Mathematics, ETH, 8092 Zurich. Email: jan@ucw.cz.}
}

\begin{document}
    \maketitle
%add:references; spanning ; the question of cherries?
\begin{abstract}
    Let $c$ be an edge-coloring of the complete $n$-vertex graph $K_n$. The
problem of finding properly colored and rainbow  Hamilton cycles in $c$ was
initiated in 1976 by Bollob\'as and Erd\H os and has been extensively studied since then. Recently it was
extended to the hypergraph setting by Dudek, Frieze and
Ruci\'nski~\cite{dudek}. We generalize these results, giving sufficient local
(resp. global) restrictions on the colorings which guarantee a properly
colored (resp. rainbow) copy of a given hypergraph $G$. 
    
    We also study multipartite analogues of these questions. 
We give (up to a constant factor) optimal sufficient conditions for a coloring $c$ of the
complete balanced $m$-partite graph to contain a properly colored or rainbow copy of a
given graph $G$ with maximum degree $\Delta$. Our bounds exhibit
a surprising transition in the rate of growth, showing that the problem is
fundamentally different in the regimes $\Delta \gg m$ and $\Delta \ll m$
Our main tool is the framework of Lu and Sz\'ekely for the space of random bijections, which we extend to product spaces.
\end{abstract}

\section{Introduction}
    In this paper, we study a problem of finding a copy of a given graph/hypergraph $G$ in another edge-colored graph/hypergraph, such that the colors of the edges of $G$ satisfy certain restrictions. A very general result of this
type is the canonical Ramsey theorem.
    
    \begin{THM}[\cite{rado}] \label{canonical} For every graph $G$, there
exists an integer $n$ such that any coloring of the edges of the complete
graph $K_n$ contains at least one of the following copies of $G$:
        
        \begin{enumerate} \itemsep-5pt \vspace{-10pt}  
					\item a monochromatic copy, in which all the edges have the same color,
					\item a rainbow copy, in which no two edges have the same color, or
					\item a lexicographic copy, i.e.~a copy whose vertices can be ordered in such a way that the color of any edge is entirely determined by the smaller endpoint.
        \end{enumerate}
    \end{THM}
    If we restrict the number of colors, the conclusion reduces to (i), i.e.
we find a monochromatic copy of a graph $G$. It is natural to ask what are
possible restrictions that guarantee (ii). To tackle this question, we give the
following definitions.
    
    Let $H$ be a graph. Throughout the paper, by a \emph{coloring of $H$}, we
mean an edge-coloring $c: E(H) \rightarrow \mathbb{N}$. A coloring $c$ is
\emph {locally $k$-bounded} if each vertex in $V$ is incident to at most $k$
edges of any given color. Furthermore, a coloring $c$ is \emph{globally
$k$-bounded} if it has  at most $k$ edges of the same color.  In both cases,
$k$ gives a lower bound on the number of colors used by $c$, so our
restrictions are in a sense reciprocal to the number of colors. Note that a
locally $1$-bounded coloring of $H$ is exactly a proper coloring (where any
two incident edges carry distinct colors), and a globally $1$-bounded
coloring is a rainbow coloring of $H$ (in which all edges receive distinct
colors).

    For a given graph $G$, we say a coloring $c$ of $H$ is \emph{G-proper} if
it contains a properly colored copy of $G$, and \emph{G-rainbow} if it
contains a rainbow copy of $G$.
A conjecture of Bollob\'as and Erd\H os from~\cite{bollobas} states that
any locally $\floor{\frac{n}{2}}$-bounded coloring of $K_n$ contains a
properly colored Hamilton cycle (denoted by $C_n$). In \cite{bollobas} they
proved a weaker result, that for $\alpha = \frac{1}{69}$, any locally $\alpha
n$-bounded coloring is $C_n$-proper. After several improvements of the
constant $\alpha$ (see  \cite{chen}, \cite{shearer} and \cite{alongutin}), the
asymptotic variant of the conjecture was proved by Lo \cite{lo}, that is, for
$\alpha < \frac 12$, any locally $\alpha n$-bounded coloring is $C_n$ proper.
    
    For general graphs $G$, it is natural to ask what is the minimum value of
$k$ for which any locally $k$-bounded coloring is $G$-proper. In particular, how does
this $k$ depend on the maximum degree of $G$? The intuition behind this
question is that the easiest way to avoid a properly colored copy of $G$ is to
forbid an embedding of its vertex of maximum degree. 
    Alon et al. \cite{alon} have shown that the coloring is certainly
$G$-proper for $k = \frac{\sqrt{n}}{\Delta(G)^{13.5}}$. This has been
significantly improved by B\"ottcher, Kohayakawa and Procacci \cite{boettcher}.
    \begin{THM}[\cite{boettcher}]  \label{boettcher1}
	If $G$ is an $n$-vertex graph with maximum degree $\Delta$, then any
locally $\left( \frac{n}{22.4 \Delta^2} \right)$-bounded coloring of $K_n$ is
$G$-proper.
    \end{THM}
    %to $k = \frac{n}{22.4 \Delta^2}$, which was shown to be best possible by the second two authors \cite{volec}. 

% and further generalised it to embed graphs with a given number of cherries.\\
    We have already hinted that global bounds on colorings yield rainbow
copies of $G$, so all the stated results have parallels in this setting. An
analogue of the Bollob\'as-Erd\H os conjecture was proposed in 1986 by Hahn and
Thomassen \cite{hahn}. They conjectured that there is a constant $\alpha$ such
that any globally $\alpha n$-bounded coloring of $K_n$ is $C_n$-{rainbow}, which was
proved by  Albert, Frieze and Reed \cite{albert} for $\alpha = \frac{1}{64}$.
    
    Using the same technique as for Theorem \ref{boettcher1}, B\"ottcher,
Kohayakawa and Procacci have translated this result to any bounded-degree
graph. This confirms a conjecture of Frieze and Krivelevich \cite{frieze},
which was originally only stated for $G$ being a tree.
    \begin{THM}[\cite{boettcher}] \label{boettcher2}
        If $G$ is an $n$-vertex graph with maximum degree $\Delta$, then any globally $\left( \frac{n}{51\Delta^2} \right)$-bounded coloring of $K_n$ is $G$-rainbow.
    \end{THM}
 In this paper we generalize Theorems \ref{boettcher1} and \ref{boettcher2} in two different directions. 
    
             \subsection{Multipartite graphs}
Motivated by a question of Oriol Serra \cite{serra} asking how
Theorems~\ref{boettcher1} and~\ref{boettcher2} adapt to the bipartite
setting~\cite{serra}, in Section~\ref{sec:partite} we study colored subgraphs in bipartite and,
more generally, multipartite graphs.
We consider $k$-bounded colorings of the complete $m$-partite
graph with $n$ vertices in each class, which we denote by $\Kmn$, and
investigate which subgraphs they contain.
In analogy with Theorems~\ref{boettcher1} and~\ref{boettcher2},
we focus on properly colored and rainbow subgraphs with maximum
degree $\Delta$.
Our results show that for a fixed value of $m$, the dependency between
$k$ and $\Delta$ exhibits a surprising discontinuous behavior when
$\Delta=\Theta(m)$.
The bipartite case (that is, when $m=2$) of these questions is of particular interest due to the following relation to Latin transversals.

A \emph{Latin square} $L$ is an $n \times n$ matrix with entries in $[n]$ such
that each row and each column contain each symbol exactly once.  A \emph{Latin
transversal} in $L$ is a transversal whose cells contain $n$ different 
symbols.  Notice that an $n \times n$ Latin square corresponds to a proper edge
coloring of $K_{n, n}$ with $n$ colors, and a Latin transversal in it
corresponds to a rainbow perfect matching. A famous conjecture of Ryser
\cite{ryser} from 1967 states that every $n \times n$ Latin square for odd $n$
contains a Latin transversal.  As a step towards this conjecture, Erd\H os and
Spencer \cite{spencer} showed that any globally $\frac{n-1}{4e}$-bounded
coloring of $K_{n,n}$ contains a rainbow perfect matching.
In \cite{perarnau}, Perarnau and Serra studied the Latin-transversals problem,
using the framework of Lu and Sz\'ekely for applying local lemma to random injections~\cite{lu}.
One of their results gives an asymptotic count of the rainbow matchings in
globally bounded colorings of $K_{n,n}=K_{2\otimes n}$. Here we study the problem of finding rainbow or properly colored
copies of various graphs in edge-colored $\Kmn$ for $m\ge2$.
Our first result in this direction is the following.
	
     \begin{THM} \label{partite2}
            Suppose $c$ is a globally $\left( \frac{n}{110 \Delta^2} \right)$-bounded coloring of $\Kmn$, and let $G$ be an $m$-partite graph with a partition $V(G) = U_1 \cup U_2 \cup \dots \cup U_m$ satisfying $|U_i| \leq n$, and maximum degree $\Delta$. Then $c$ is $G$-rainbow.
	 \end{THM}

 %!!       Note that  the definitions and hypotheses allow us to reorder the parts of $G$ and move vertices between the parts. However, the statement we actually prove is stronger - we can embed $G$ when its parts are assigned corresponding parts of $\Kmn$.
We also prove a similar result for properly colored copies of $G$ in locally bounded colorings of $\Kmn$.
    \begin{THM} \label{partite1}
            Suppose $c$ is a locally $\left( \frac{n}{48 \Delta^2} \right)$-bounded coloring of $\Kmn$, and let $G$ be an $m$-partite graph with a partition $V(G) = U_1 \cup U_2 \cup \dots \cup U_m$ satisfying $|U_i| \leq n$, and maximum degree $\Delta$. Then $c$ is $G$-proper.
	 \end{THM}
	    Conversely, we give graphs $G$ and a coloring $c$ simultaneously showing that both statements are optimal up to a constant factor: the coloring is globally bounded, and it still does not contain even a properly colored copy of $G$. Notice that the order of $G$ in the following proposition is fixed (independent on $n$).%and the hypothesis does not require its parts to be preserved.
	    
	 \begin{PROP} \label{partite3}
	    Suppose $m \geq 2$, $q $ is a prime power, and $n \geq 3q^2$. %!! could be 2q^2 +2q + 1
	        There exists an $m$-partite graph $G$ with at most $3q^2$ vertices in each part, maximum degree $\Delta \leq q + 2m$, and a globally $\left(\frac{n}{q^2+q}\right)$-bounded coloring of $\Kmn$ which does not contain a properly colored copy of $G$.
	 \end{PROP}
	 
	 When $\Delta >> m $, the coloring is $O \left(\frac{n}{\Delta^2} \right)$-bounded, matching the bounds of Theorem \ref{partite2} and \ref{partite1}. On the other hand, for $\Delta << m$ the problem becomes fundamentally different. By viewing $\Kmn$ as an almost complete graph, we show that even $O \left( \frac{n}{\Delta}\right)$-bounded colorings contain the required %corresponding, appropriate
	 copies of $G$. This embedding result is also matched by the corresponding construction, and it follows from the results of the second and the third author~\cite{volec}.
	 %Note that a graph $G$ is trivially $(\Delta(G)+1)$-colorable (greedy algorithm), so it makes sense that for $m > \Delta$ this info no longer plays the crucial role. In the regime $m = \Theta(\Delta)$, the bounds in Theorems \ref{partite1}, \ref{partite2} coincide with those in \ref{partite4}.
    \begin{THM}[\cite{volec}] \label{partite4}
    There exist constants $\alpha$ and $\beta$ such that the following holds. Let $G$ be an $N$-vertex graph of max degree $\Delta$, and $K$ an $N$-vertex graph of minimum degree $N\cdot\left(1 - O\left(\Delta^{-1} \right) \right)$.
        \begin{enumerate} \itemsep-5pt \vspace{-10pt}  
					\item Any locally $ \left(\frac{N}{\alpha \Delta^2} \right)$-bounded coloring of $K$ is $G$-proper.
					\item Any globally $\left( \frac{N}{\beta \Delta^2} \right)$-bounded coloring of $K$ is $G$-rainbow.
					%!! shall we put the two statements together/
					%	\item For every prime power $q$ and integer $mn$, there exists a $(q^2+q+1)$-vertex graph $G$ with maximum degree $\Delta = q+1$, and a locally $\left(\frac{3.9mn}{\Delta^2} \right)$-bounded coloring $c$ of $E(\Kmn)$ which is not $G$-proper.
					%	\item For any integers $\Delta$, $m$ and $n$ such that $\Delta$ is even and $\left(\frac{\Delta}{2}+1 \right)^2$ divides $mn$, there exists an $mn$-vertex $m$-partite graph $G$ with maximum degree $\Delta$ and a globally $\left(\frac{16n}{\Delta^2} \right)$-bounded coloring $c$ of $E(\Kmn)$ which is not $G$-rainbow. 
        \end{enumerate}
    \end{THM}
        %The statements can be found in the concluding remarks from Sudakov and Volec \cite{volec}. It follows, 
        By applying (i) to $\Delta \leq \delta m$ (where $\delta$ is a small constant), $N =mn$ and $K=\Kmn$, we get that that any locally $\left( \frac{mn}{\alpha \Delta^2}\right)$-bounded edge-coloring of $\Kmn$ is $G$-proper. The analogue is true for rainbow copies of $G$ in globally bounded colorings of $\Kmn$, using (ii). It is shown in \cite{volec} that the bounds are optimal up to a constant factor.
        
     We emphasize that  the definitions and hypotheses in all the theorems do not fix any particular partition or ordering of the parts of $G$. 
     
         \subsection{Properly colored and rainbow copies of hypergraphs}      
        The problem of Bollob\'as and Erd\H os on finding properly colored Hamilton cycles in locally bounded colorings extends naturally to hypergraphs and has recently been studied in \cite{dudek} and \cite{ferrara}.
        We will be looking at edge colorings of $r$-uniform hypergraphs $H$, that is, assignments $c: E(H) \rightarrow \mathbb{N}$. A subhypergraph $G$ of $H$ is said to be \emph{properly colored} if any two overlapping edges of $G$ receive different colors. Furthermore, if every edge of $G$ receives a different color, the subgraph $G$ is \emph{rainbow}.
        We impose the same type of restrictions on the colorings $c$.
          A coloring $c$ is \emph{locally $k$-bounded} if the hypergraph induced by a single color has maximum degree at most $k$, which means that each vertex is contained in at least $k$ edges of a particular color. Formally,  $\Delta\left( H\left[ c^{-1}(i) \right] \right) \leq k$ for all $i\in \mathbb{N}$.  
        We say that $c$ is \emph{globally $k$-bounded} if each color is used at most $k$ times.
        
        Dudek, Frieze and Ruci\'nski \cite{dudek} have studied the existence of properly colored and rainbow Hamilton cycles in colored complete $r$-uniform hypergraphs. There are several different notions of hypergraph cycles. For  $\ell \in [r-1]$, an \emph{$\ell$-overlapping cycle}  $C_n^{(r)}(\ell)$ is an $n$-vertex hypergraph in which, for some cyclic ordering of its vertices, the edges consist of $r$ consecutive vertices, and each two consecutive edges share exactly $\ell$ vertices. For $r=2$ and $\ell=1$, this reduces to the graph cycle. The two extreme cases, $\ell=1$ and $\ell=r-1$, are usually referred to as \emph{loose} and \emph{tight} cycles respectively. 
        
        Given an $n$-vertex $r$-graph $H$, any subgraph of $H$ isomorphic to $\cnrl$ is called an \emph{$\ell$-overlapping Hamilton cycle}. It is easy to show that $\cnrl$ has precisely $\frac{n}{r-\ell}$ edges and therefore we cannot expect $H$ to have one unless $r-\ell$ divides $n$. Generalising the result of Bollob\'as and Erd\H os from~\cite{bollobas}, Dudek, Frieze and Ruci\'nski~\cite{dudek} have shown the following (see also \cite{ferrara} for some further results).
        \begin{THM}[\cite{dudek}] \label{thmdudek}
            For every $\ell \in [r-1]$ there is a constant $\alpha$ (resp. $\beta$) such that if $n$ is sufficiently large and divisible by $r-\ell$, then any locally $\alpha n^{r-\ell}$-bounded (resp. globally $\beta n^{r-\ell}$-bounded) coloring of $K_n^{(r)}$ is $\cnrl$-proper (resp. -rainbow).
        \end{THM}
        
        Their proof relies heavily on  the cyclic structure of $\cnrl$.  We show that, as in the result of Alon et al.~\cite{alon}, it is actually sufficient to impose restrictions on the degrees in $G$. For a set $S \subset V(H)$, we say the \emph{degree}, or more accurately, $|S|$-degree of $S$ is the number of edges of $H$ containing $S$, denoted by $d(S)$ or $d_H(S)$. For a given $\ell \in \{0,1 , \dots, r-1\}$, the maximum $\ell$-degree of $H$ is the maximum degree over all vertex-subsets of order $\ell$, that is, $\Delta_\ell(H)
        := \max \{d_H(S): S \subset V(H),\, |S|=\ell \}$. For example, the $\ell$-overlapping $r$-uniform cycle has $\Delta_{\ell}=2$ and $\Delta_{\ell+1} = 1$. Note that $\Delta_0(H)$ is just the number of edges of $H$.
        We apply the Lov\'asz Local Lemma, or more specifically, the corresponding framework of Lu and Sz\'ekely \cite{lu} 
        in order to generalize Theorem \ref{thmdudek} to hypergraphs $G$ with bounded maximum $\ell$-degrees.

        \begin{THM} \label{thmhyp1} For a given uniformity $r$ and $\ell \in [r-1]$, there exist positive constants $c_1$ and $c_2$ such that the following holds. Let $G$ be an $r$-uniform hypergraph on at most $n$ vertices satisfying $\Delta_{\ell+1}(G) =1$. Then the following holds.
            \begin{enumerate} \vspace{-10pt} \itemsep-1pt
                \item Any locally $\frac{c_1 n^{r-\ell}}{\Delta_1(G)\Delta_{\ell}(G)}$-bounded coloring of $\Knr$ is $G$-proper.
                \item Any globally $\frac{c_2 n^{r-\ell}}{\Delta_1(G)\Delta_{\ell}(G)}$-bounded coloring of $\Knr$ is $G$-rainbow.
            \end{enumerate}
        \end{THM}
        The proof of Theorem~\ref{thmhyp1} is given in Section \ref{sec:hypergraphs}, where
        we also show that the dependencies on $n$ for both parts of the theorem are the best possible. For $l=r-1$, we were also able to show that the dependence on the maximum degrees of $G$ is best possible.
            %Let $D^{(r)}(l)$ be an $m$-vertex $r$-uniform hypergraph such that each set of $l+1$ vertices is contained in exactly one edge.
%            
%            Say for which m it exists, maybe quote thm..
%            Note that $D^{(r)}(l)$ satisfies  $\Delta_{l+1} (D^{(r)}(l))= 1$ and each $l$-subset is contained in $m-l$ edges.
%        \begin{THM} \label{besthyp}
%            There exists a globally $n^{r-l}$-bounded coloring of $\Knr$ which contains no  properly colored copy of $D^{(r)}(l)$.
%        \end{THM}
%            In fact, our coloring contains no  $r$-uniform subgraph $G$ with the property $d_G(S) \geq 2$ for all $S \subset V(G)$ with $|S| = l$.
%     

\section{Lov\'asz Local Lemma and Lu-Sz\'ekely framework for random injections}
Probabilistic methods are very useful for       
constructing combinatorial objects satisfying certain properties. The idea is to show that an object chosen at random satisfies the properties in question with positive probability. This is exactly the statement of the Lov\'asz Local Lemma, and the sufficient conditions are certain mutual correlations between the desired properties.
    
    The Lemma is usually formulated in terms of \emph{bad events} $B_1,\, B_2, \dots,\, B_N$, which correspond to the undesired properties of our object. We say that a graph $D$ with the vertex set $[N]$ is a \emph{dependency graph} for a family of events $\mathcal{B}= \{B_1, \dots, B_N \}$ if for every $i \in [N]$, the event $B_i$ is mutually independent of all the events $B_j \neq B_i$ such that $ij \notin E(D)$. More generally, $D$ is a \emph{negative dependency graph} for $\mathcal{B}$ if for every $i \in [N]$ and every set $J \subset \{j : ij \notin D \}$, it holds that $\pr{B_i \mid \bigwedge_{j \in J} \overline{B_j}} \leq \pr{B_i}$.
    
    The original version of the local lemma, due to Erd\H{o}s and Lov\'asz \cite{lovasz}, used a dependency graph for the set of bad events in order to control the correlations. It was then observed by Erd\H{o}s and Spencer \cite{spencer} that in fact the same proof applies when we capture the correlations using a negative dependency graph. They called this variant the \emph{lopsided Lov\'asz Local Lemma}. We use the following version of the Lemma, often called the Asymmetric Local Lemma. It is %stated and applied  by the second two authors in \cite{volec}, and
    proved e.g. in \cite[Chapter 19.3]{molloy} in the non-lopsided form.
    
    \begin{LEMMA}[Asymmetric Lopsided Lov\'asz Local Lemma] \label{asymmetric}
        Let $\mathcal{B} = \{B_1, \dots, B_N \}$ be a set of bad events with a negative dependency graph $D = ([N], \mathcal{E})$. If for all $i \in [N]$,
        $\pr{B_i} \leq \frac{1}{4} \ \text{and} \ \sum_{ij \in \mathcal{E}}\pr{B_j} \leq \frac{1}{4},$
    %$$\pr{B_i} \leq \frac{1}{4} \quad \text{and} \quad \sum_{ij \in \mathcal{E}}\pr{B_j} \leq \frac{1}{4},$$
    then $$\displaystyle \pr{\bigwedge_{i \in [N]} \overline{B_i}}>0.$$
    \end{LEMMA}
    
    We will be using a type of negative dependency graph which is specific to the probability spaces of random injections. It is a slight generalization of a dependency graph first constructed by Lu and Sz\'ekely \cite{lu}. However, we cannot quote their results directly because our probability space is slightly more general. Let $X = X_1 \times \dots \times X_m$ and $Y = Y_1 \times \dots Y_m$, where the parts $X_i$ and $Y_i$ satisfy $|X_i| \leq|Y_i| $. Consider the probability space $\Omega$ generated by picking uniformly random injections $\sigma_i: X_i \rightarrow Y_i$ for $i \in [m]$, and setting $\sigma = \sigma_1 \cup \dots \cup \sigma_m$ to be the induced injection between $X$ and $Y$. Denote the set of such injections by $\mathcal{S}$.
    
     Let $\tau: T \rightarrow U$ be a given bijection between $T \subset X$ and $U \subset Y$. The corresponding \emph{canonical event} $B$  consists of all part-respecting injections $X \rightarrow Y$ which extend  $\tau$, that is
    $$B = \Omega(T,U, \tau) :=\{\sigma \in \mathcal{S} : \sigma(x) = \tau (x) \text{ for all }x \in T\}. $$
    %for some $T \subset X, U \subset Y$ satisfying $|T|=|U|$ and a part-respecting bijection $\tau: T \rightarrow U$.
Two events $\Omega(T_1, U_1, \tau_1)$ and $\Omega(T_2, U_2, \tau_2)$ \emph{$\mathcal{S}$-intersect} if the sets $T_1$ and $T_2$ intersect or $U_1$ and $U_2$ intersect. A result of Lu and Sz\'ekely \cite{lu} implies that for a set of bad canonical events, the graph whose vertices are the bad events and edges connect exactly the $\mathcal{S}$-intersecting events is a negative dependency graph. Although Lu and Sz\'ekely considered the case $m=1$, their result can be straightforwardly generalized for arbitrary $m$.
    
    \begin{THM} \label{luszekely} %!!lemma?
        Let $\Omega$ be the probability space generated by picking an injection from $\mathcal{S}$ uniformly at random, with the notation defined above. Furthermore, let $\mathcal{B} = \{ B_1,\, B_2, \dots B_N \}$  be some family of canonical events in $\Omega$ and let $D$ be a graph on vertex set $[N]$ and $ij \in E(D)$ if and only if the events $B_i$ and $B_j$ $\mathcal S$-intersect. It holds that $D$ is a negative dependency graph.
    \end{THM}
     
    For the sake of completeness, we present a proof of Theorem \ref{luszekely} in the Appendix. The dependency graph treated there and originally proposed by \cite{lu} is even a subgraph of $D$, that is, the actual theorem is slightly more general.
Closely related generalizations of the original results of Lu and Sz\'ekely
formulated in the language of hypergraph matchings have been proven in~\cite{lumohrszek} and~\cite{cano}.
We would also like to mention that generalizations analogous to Theorem~\ref{luszekely} were recently studied from an algorithmic point of view by Harris and Srinivasan~\cite{harris}, and by Harvey and Vondr\'ak \cite{harvey}.

\section{Properly colored and rainbow copies of $m$-partite graphs} \label{sec:partite}
We start our exposition by studying bounded colorings of $\Kmn$, where we seek colored
copies of $m$-partite graphs with maximum degree $\Delta = \Omega(m)$.
In this regime, we heavily rely on the $m$-partite structure of $G$ and $\Kmn$,
and as the main tool apply a multidimensional version of the framework of Lu and Sz\'ekely.
%{\color{red} We remark that in both theorems \ref{partite2} and \ref{partite1}, it suffices to assume that for each vertex $v$ of $G$ and part $U_i$, $|\Gamma(v) \cap U_i| \leq \Delta$. In words, it suffices to bound the degrees of vertices into each part of $G$, which yields a slightly stronger theorem.
%}
Throughout the section, we omit the floor  and ceiling signs whenever it is not critical.

For a slight convenience, we deal first with properly colored subgraphs in locally bounded colorings.
    \begin{proof}[Proof of Theorem \ref{partite1}]
        Let $G$ be an $m$-partite graph on vertex set $U = U_1 \cup U_2 \cup \dots U_m$ such that $|U_i| \leq n$ and no part $U_i$ contains an edge. Let $\Delta$ be the maximum degree of $G$, and $c$ a locally $k$-bounded edge-coloring of $\Kmn$, where $k = \frac{n}{48\Delta^2}$. We take the vertex set of $\Kmn$ to be $V = V_1 \cup V_2 \cup \dots V_m$ with $|V_i|=n$ for all $i$. We claim that there exists a properly colored copy of $G$ in $\Kmn$, even with the additional constraint that each part $U_i$ is mapped into $V_i$.
        
        Let $\mathcal S$ be the set of injections $f^\prime: U\rightarrow V$  satisfying $f^\prime(U_i) \subset V_i$ for all $i\in[m]$, $\Omega$ the uniform probability space on $\mathcal S$, and $f$ a random injection drawn from $\Omega$. To index the bad events, we set up some notation. As first, fix an ordering of $U$. A triple $u_1 \mathyp u_2 \mathyp u_3$ of distinct vertices of $G$ denotes a \emph{cherry} in $G$, that is a path of length two with the middle vertex $u_2$. We will only be considering cherries $u_1 \mathyp u_2 \mathyp u_3$ with $u_1 < u_3$. Similarly, $[v_1 v_2 v_3]$ denotes a monochromatic cherry in colored $\Kmn$, that is, a triple for which $v_1v_2$ and $v_2 v_3$ are edges that satisfy $c(v_1 v_2) = c(v_2 v_3)$. Note that we only require $v_1 \neq v_3$ and not an ordering between them, so that the bijection that maps $u_i$ to $v_i$ for $i = 1, 2 ,3$ is counted exactly once.
        
        The bad events will be all events of form
        $$B_{u_1 \mathyp u_2 \mathyp u_3}^{[v_1 v_2 v_3]} = \{f \in {\mathcal S}: f(u_i) = v_i \text{ for } i= 1,2,3 \},$$
        for all choices of cherries $u_1 \mathyp u_2 \mathyp u_3$ and $[v_1 v_2 v_3]$ satisfying the conditions above. The set of bad events is denoted by $\mathcal{B}$.
        
        As granted by Lemma \ref{luszekely}, the graph on vertex set $\mathcal{B}$ and edges between $B_{u_1 \mathyp u_2 \mathyp u_3} ^{[v_1 v_2 v_3]}$ and $B_{u_1^\prime \mathyp u_2^\prime \mathyp u_3^\prime}^{[v_1^\prime v_2 ^\prime v_3^\prime]}$ whenever the two events $\mathcal S$-intersect is a negative dependency graph. By definition, this occurs only when the corresponding cherries $\{ u_1,  u_2, u_3\}$ and $\{ u_1^\prime,  u_2^\prime, u_3^\prime\}$, or  $\{ v_1,  v_2, v_3\}$ and $\{ v_1^\prime,  v_2^\prime, v_3^\prime\}$ intersect. If the prior occurs, we call the events $G$-intersecting, and otherwise we call them $K$-intersecting.
        %!! generalise this definition
        
        Each bad event $B = B_{u_1 \mathyp u_2 \mathyp u_3} ^{[v_1 v_2 v_3]}$ satisfies $\pr{B} \leq \frac{1}{n^2(n-1)} < \frac{1}{4}$, since there are always $n$ possibilities for choosing the image of $u_2$, and at least $n(n-1)$ possibilities for embedding the leaves $u_1$ and $u_3$ (as they might lie in the same part $V_i$). By Lemma \ref{asymmetric} it remains to prove for $B \in \mathcal{B}$
        \begin{equation} \label{sumpb1}
            \sum_{\substack{
            B^\prime \in \mathcal{B} \\
            B^\prime \ {\mathcal S}\text{-intersects } B
            }} \pr{B^\prime} \leq \frac{1}{4}.
        \end{equation}
        Upon showing that, with positive probability none of the bad events occur, in which case $f$ yields a properly colored embedding of $G$ into $\Kmn$.
        
        It remains to count the $\mathcal S$-intersecting events. Fix an event $B = B_{u_1 \mathyp u_2 \mathyp u_3}^{[v_1 v_2 v_3]}$. First we count the number $I_G(B)$ of events $B^\prime = B_{u_1^\prime \mathyp u_2^\prime \mathyp u_3^\prime}^{[v_1^\prime v_2 ^\prime v_3^\prime]}$ which are $G$-intersecting with $B$. Without loss of generality, $u_1 \in \{ u_1^\prime,  u_2^\prime, u_3^\prime\}$ (note that we allow more than one vertex in the intersection).
        We have two cases:
        \begin{enumerate} \itemsep-5pt \vspace{-10pt}
            \item If $u_1$ is a leaf of $u_1^\prime \mathyp u_2^\prime \mathyp u_3^\prime$, then we have $\Delta$ choices for the apex $u_2^\prime$ and $\Delta-1$ choices for the second leaf as a neighbor of $u_2^\prime$ in $G$. The ordering of the cherry $u_1^\prime \mathyp u_2^\prime  \mathyp u_3 ^\prime$ is then fixed by the requirement $u_1 ^\prime < u_3 ^\prime$.
            \item If $u_1$ is the apex $u_2^\prime$, then there are $\frac{\Delta(\Delta-1)}{2}$ choices for the two leaves, whose ordering is again predetermined.
        \end{enumerate} \vspace{-10pt}
        Altogether, this gives $\frac{3}{2}\Delta(\Delta-1)$ choices for $u_1^\prime \mathyp u_2^\prime \mathyp u_3^\prime$. Each cherry in $G$ can be mapped to at most $n^2 k$ monochromatic cherries in $\Kmn$ in a part-respecting manner - there are $n^2$ ways to choose $v_1^\prime$ and $v_2^\prime$ inside the parts corresponding to $u_1^\prime$ and $u_2^\prime$, and then further $k$ choices of $v_3^\prime$ satisfying $c(v_1^\prime v_2 ^\prime) = c(v_2^\prime v_3^\prime)$. Summing up and multiplying by 3 to account for the fact that the intersection may occur at $u_1$, $u_2$ or $u_3$, we conclude 
        \begin{equation} \label{igb}
            I_G(B) \leq \frac{9}{2} \Delta(\Delta-1)n^2 k.
        \end{equation}
        
        Next, denote the number of events which $K$-intersect $B$ by $I_K(B)$. As before, let $B^\prime = B_{u_1^\prime \mathyp u_2^\prime \mathyp u_3^\prime}^{[v_1^\prime v_2 ^\prime v_3^\prime]}$ be such an event and suppose $v_1 \in \left \{v_1^\prime,\, v_2 ^\prime,\, v_3^\prime \right \}$. There are $n$ ways of choosing the preimage of $v_1$ in the corresponding part of $G$. Again, we distinguish two cases.
        \begin{enumerate} \itemsep-5pt \vspace{-10pt}
            \item If $v_1 = v_2^\prime$ is the apex of $[v_1^\prime v_2 ^\prime v_3^\prime]$, then there are $\frac{\Delta(\Delta-1)}{2}$ ways of choosing the leaves of $u_1^\prime$ and $u_3^\prime$.
            \item If $v_1$ is a leaf of $[v_1^\prime v_2 ^\prime v_3^\prime]$, then there are $\Delta(\Delta-1)$ ways to complete the preimage of $v_1$ into a cherry ${u_1^\prime \mathyp u_2^\prime \mathyp u_3^\prime}$ in $G$ and the condition $u_1^\prime < u_3^\prime$ determines whether $v_1 = v_1^\prime$ or $v_1 = v_3^\prime$.
        \end{enumerate} \vspace{-10pt}
        Having chosen ${u_1^\prime \mathyp u_2^\prime \mathyp u_3^\prime}$, there are $nk$ ways to complete the monochromatic cherry $[v_1^\prime v_2 ^\prime v_3^\prime]$ in $\Kmn$, which gives at most $\frac{3}{2}\Delta(\Delta-1)n^2k$ bad events $\mathcal S$-intersecting $B$ at $v_1$. Since the intersection can also occur at $v_2$ and $v_3$, the bound is
        \begin{equation} \label{ikb}
            I_K(B) \leq \frac{9}{2} \Delta(\Delta-1)n^2 k.
        \end{equation}
        
        Introducing these bounds, using $\frac{n}{n-1}\leq \frac 43$ for $n \geq 4$ and $k \leq \frac{n}{48\Delta^2}$, we get
        $$
            \sum_{\substack{
            B^\prime \in \mathcal{B} \\
            B^\prime \ {\mathcal S}\text{-intersects } B
            }} \pr{B^\prime} \leq 2\cdot \frac{9}{2} \Delta(\Delta-1)n^2 k \cdot \frac{1}{n^2(n-1)} < \frac{12\Delta^2 k}{n} \leq \frac{1}{4},
        $$
        which proves \eqref{sumpb1}.
     \end{proof}
    We continue with the proof of our result on rainbow subgraphs in globally bounded colorings.
		\begin{proof}[Proof of Theorem \ref{partite2}]
				We follow the outline of the previous proof, but we now have to avoid the events that any two edges in our embedding of $G$ carry the same color. Recall, $G$ is an $m$-partite graph on vertex set $U = U_1 \cup U_2 \cup \dots U_m$, that is, parts $U_i$ are independent sets with $|U_i| \leq n$. The maximum degree of $G$ is denoted by $\Delta$. Let $c$ be a globally $k$-bounded edge-coloring of $\Kmn$, where $k = \frac{n}{110 \Delta^2}$. We take the vertex set of $\Kmn$ to be $V = V_1 \cup V_2 \cup \dots V_m$ with $|V_i|=n$ for all $i$. We claim that there exists a rainbow copy of $G$ in $\Kmn$, even with the additional constraint that each part $U_i$ is mapped into $V_i$.
        
        As before, let $\mathcal{S}$ denote the set of injections $f^\prime: U\rightarrow V$  satisfying $f^\prime(U_i) \subset V_i$ for all $i$. Let $f$ be an injection chosen uniformly at random from $\mathcal{S}$. The set of bad events is now extended to $\mathcal{B} \cup \mathcal{C}$, where $\mathcal{B}$ and $\mathcal{C}$ are as follows. As before,  $\mathcal{B}$ is the set of events of form $\baduv$. Recall that $\cheg$ is a cherry in $G$ with apex $u_2$ and leaves satisfying $u_1 < u_3$, and $\chek$ is a monochromatic cherry in $\Kmn$ with no specified ordering between the leaves.

				Similarly, a \emph{quadruple} $\quag$ in $G$ denotes two \emph{disjoint} edges $u_1u_2$ and $u_3 u_4$ in $G$ such that $u_1 < u_2$, $u_3 < u_4$ and $u_1 < u_3$. A monochromatic \emph{quadruple} in $\Kmn$ is a 4-tuple  $[v_1 v_2 v_3 v_4]$  of distinct vertices $v_i \in V$ satisfying $c(v_1 v_2) = c(v_3 v_4)$.  Subject to such choices of $u_i$ and $v_i$, we define  $\mathcal{C}$ be the set of events of form
				$$\caduv = \{f \in {\mathcal S}: f(u_i) = v_i \text{ for } i= 1,\dots, 4 \}.$$
				
			%We denote the set of all events $\caduv$ constructed as above by $\mathcal{C}$. We note down the following observation to use in the counting.\\
			%\textbf{Fact 1.} Given a pair of edges $u_1 u_2$ and $u_3u_4$ in $G$ with their respective images $v_1, v_2, v_3, v_4 \in V(K)$ satisfying $c(v_1 v_2) = c(v_3 v_4)$, the corresponding event $B \in \mathcal{C}$ is uniquely determined by the ordering of vertices in $G$.\\

        As granted by Lemma \ref{luszekely}, the graph on vertex set $\mathcal{B} \cup \mathcal{C}$ and edges between $B$ and $B^\prime $ whenever the two events $\mathcal S$-intersect is a negative dependency graph. By definition, this occurs only when the cherries or quadruples corresponding to $B$ and $B^\prime$ intersect.
        
				Just like before, each event $B$ in $\mathcal{B}$ satisfies $\pr{B} \leq \frac{1}{n^2(n-1)} < \frac{1}{4}$, whereas events $B \in \mathcal{C}$ satisfy a stronger inequality $\pr{B} \leq \frac{1}{n^2 (n-1)^2} < \frac{1}{4}$. Equality is attained when two pairs of vertices lie in the same part of $G$, say $u_1, u_3 \in U_i$ and $u_2, u_4 \in U_j$, and otherwise the probability is strictly smaller. By Lemma \ref{asymmetric} it remains to prove for $B \in \mathcal{B} \cup \mathcal{C}$,
        \begin{equation} \label{sumpb3}
            \sum_{\substack{
            B^\prime \in \mathcal{B} \cup \mathcal{C} \\
            B^\prime \; {\mathcal S}\text{-intersects } B
            }} \pr{B^\prime} \leq \frac{1}{4}.
        \end{equation}
        Upon showing that, with positive probability none of the bad events occur, in which case $f$ yields a rainbow embedding of $G$ into $\Kmn$.
        
       Consider a bad event $B \in \mathcal{B} \cup \mathcal{C}$. 	We denote the number of events in $\mathcal{B}$ which $G$-intersect (resp. $K$-intersect) $B$ by $I_G(B)$ (resp. $I_K(B)$). Analogously, $J_G(B)$ (resp. $J_K(B)$) denotes the number of events in $\mathcal{C}$ which $G$-intersect (resp. $K$-intersect) $B$. If $B$ has the form $\baduv$, fix $u \in \{u_1,\, u_2, u_3 \}$ and $v\in \{v_1,\, v_2, v_3 \}$. Otherwise, if $B = \caduv$, fix $u \in \{u_1,\, u_2, u_3, u_4 \}$ and $v\in \{v_1,\, v_2, v_3, v_4 \}$. Either way, each of the two vertices can be chosen in at most $4$ ways (as opposed to $3$ in the previous proof). We will be counting events that $G$-intersect or $K$-intersect $B$ at $u$ or $v$ respectively, and then multiply the result by $4$ to take into account all the possibilities. The bounds we obtain
       are valid in both cases, when $B \in \mathcal{B}$ as well as $B \in \mathcal{C}$.
			
			For the events $B^\prime \in \mathcal{B}$ of form $\baduvp$ $\mathcal S$-intersecting $B$, we have the same count as in equations \eqref{igb} and \eqref{ikb}, with the increase by a factor of $\frac{4}{3}$ as explained above. It follows that
			\begin{align}
			\label{igb1}
			    I_G(B)&\leq 4 \cdot \frac{3}{2}\Delta(\Delta -1)n^2 k = 6 \Delta(\Delta -1)n^2 k \text{  \quad and}\\
			    I_K(B) &\leq 6 \Delta(\Delta -1)n^2 k.
			\end{align}
			
		To bound $J_G(B)$, fix an event $B^\prime = \caduvp$ with $u \in \{u_1^\prime, u_2^\prime, u_3^\prime, u_4^\prime \}$. In counting the choices for vertices $u'_i$ and $v'_i$, we switch back and forth between $G$ and $\Kmn$ to get the best bounds. There are $\Delta$ possible choices for the vertex $u^\prime$ neighbouring $u$, and $n^2$ choices for the images $v$ and $v^\prime$ of $u$ and $u^\prime$ in the corresponding parts of $G$. Then we fix a pair $\{\tilde{v}, \tilde{v}^\prime \}$ such that $c(\tilde{v} \tilde{v}^\prime) = c(v v^\prime)$, which can be done in $k$ ways, since the coloring is globally $k$-bounded. The preimage of $\tilde{v}$ and $\tilde{v}^\prime$ can again be chosen in  at most $n \Delta$ ways. The ordering of vertices in $G$ now uniquely determines the event $\caduvp$. Putting the numbers together, and taking into account at most four choices of $u$ from $B = \baduv$ or $B = \caduv$ gives
		\begin{equation}
		 \label{jgb}
		 J_G(B) \leq 4 n^2\Delta k n\Delta = 4\Delta^2 n^3k.
		\end{equation}
        
        To control $J_K(B)$, once again fix an event $B^\prime = \caduvp$ satisfying $v \in \left\{v_1^\prime, v_2^\prime, v_3^\prime, v_4^\prime \right \}$. We start by choosing the preimage $u$ of $v$ among the $n$ possible vertices. From there, the argument is the same as for $J_G(B)$, implying $
        J_K(B) \leq 4\Delta^2 n^3 k.$ 
        
        Summing up and using $k \leq \frac{n}{110 \Delta^2}$, we get for $n\geq 5$
        \begin{align*}
        \sum_{\substack{
            B^\prime \in \mathcal{B} \cup \mathcal{C}  \\
            B^\prime \ {\mathcal S}\text{-intersects } B
            }} \pr{B^\prime} 
%           + \sum_{\substack{
%            B^\prime \in \mathcal{C} \\
%B^\prime \text{ intersects } B
%            }} \pr{B}
            &\leq   \frac{12 \Delta(\Delta-1)n^2k}{n^2(n-1)} + \frac{8 \Delta^2 n^3 k}{n^2(n-1)^2} \\
            &< \left(12\cdot \frac{5}{4} + 8 \cdot \frac{25}{16} \right) \frac{\Delta^2 k}{n} \leq\frac{1}{4}.
        \end{align*}    
            \end{proof}
    To conclude this section, we show that Theorems \ref{partite2} and \ref{partite1} %!! ordering
    are the best we can hope for up to a constant factor. The proof relies on some ideas from \cite{volec}.
    \begin{proof}[Proof of Proposition \ref{partite3}]
        As stated, let $q \geq m$ be a prime power, and we let $V(G) = P \cup L \cup T_1 \cup T_2 \cup \dots T_{q^2+q} \cup S_1 \cup S_2 \cup \dots S_{q^2+q}$, where $P = \{p_0, p_1,\dots p_{q^2+q} \}$ and $L = \{\ell_0,\, \ell_1,\, \dots \ell_{q^2+q} \}$ correspond to the points and the lines of the finite projective plane $PG(2, q)$, respectively. Moreover, for any $i \in \left[q^2+q\right]$, both the set $T_i$ and the set $S_i$ induces a clique of order $m-1$. We join $p_0$ and $p_1$ to all the vertices of $T_1$, $p_1$ and $p_2$ to all the vertices of $T_2$, and so on, up to $p_{q^2+q-1}$, $p_{q^2+q}$ and $T_{q^2+q}$. Analogously, we connect the lines $\ell_{j-1}$ and $\ell_{j}$ to the vertices of the clique $S_j$ for every $j\in\left[q^2+q\right]$. Finally, we join $p_i$ to $\ell_j$ when the point $p_i$ belongs to the line $p_j$. Note that our construction forces any embedding of $G$ into $\Kmn$ (regardless of the coloring) to place points into the same part of $\Kmn$, and lines into the same part. Indeed, suppose there are two points $p_i, p_{i+1}$ which are placed into two different parts of $\Kmn$. Then, since they both connected to a clique $T_{i+1}$ of size $m-1$, the vertices of this clique can not be placed into the remaining $m-2$ parts.
				
				We now show that $G$ is indeed $m$-partite, i.e. we can partition the vertices of $G$ into independent sets $U_1,\dots U_m$. To split $P \cup T_1 \cup T_2 \cup \dots T_{q^2+q}$, we set $P \subset U_1$, and parts $U_2,\dots U_m$ contain one vertex of each clique $T_i$. Vertices in $L \cup  S_1 \cup S_2 \cup \dots S_{q^2+q}$ are split in the analogous way, with say $L \subset U_2$. And indeed, each part of $G$ has at most $2(q^2+q+1) \leq 3q^2$ vertices. Finally, the maximum degree of $G$ is $\Delta(G) = (q+1) + (2m-2) \leq q+2m$, as stated.
        
        The locally $\ceil{\frac{n}{q^2+q}}$-bounded coloring of $\Kmn$ which we now describe is an analogue of the coloring from \cite[Lemma 29]{volec}, motivated by the canonical colorings of $K_n$. Let $V_1, V_2, \dots V_m$ be the parts of $\Kmn$, and let each part $V_i$ be split into $q^2+q$ clusters $V_{ij}$ as evenly as possible. If $x$ and $y$ are vertices of $\Kmn$ with $x \in V_{ij}$, $y \in V_{i^\prime j^\prime}$ and $i < i^\prime$, then the edge $xy$ gets color $(x, V_{i^\prime j^\prime})$. In other words, if we order the parts $V_1,\dots, V_m$ vertically downward, the colors are indexed by $(x, V_{i^\prime j^\prime})$, where $x$ lies strictly above $V_{i^\prime}$ and each downward \emph{fan} from $x$ to a cluster $V_{i^\prime j^\prime}$ gets its own unique color. The resulting coloring $c$ is globally bounded by the order of the clusters, that is by $\frac{n}{q^2+q}$.
        
        Suppose there is a properly colored embedding $f: V(G) \longrightarrow V(\Kmn)$. Then $f(P) \subset V_i$ and $f(L) \subset V_j$, so, without loss of generality, $i < j$. By the pigeonhole principle, two lines, say $\ell_0$ and $\ell_1$, lie inside the same cluster. But then the cherry formed by $\ell_0$, $\ell_1$ and the representative of their intersection in $PG(2, q) \subset G$ is monochromatic by construction of $c$.
    \end{proof}
    
    \section{Properly colored and rainbow copies of bounded-degree hypergraphs} \label{sec:hypergraphs}
    This section is concerned with bounded edge-colorings of $r$-uniform hypergraphs. As in Theorems~\ref{boettcher1}, \ref{boettcher2}, \ref{partite2}
and \ref{partite1}, we establish upper bounds on $k$ so that any locally (globally) $k$-bounded coloring of $\Knr$ contains a properly colored (rainbow) copy of a given hypergraph $G$. This result is given in Theorem~\ref{thmhyp1}.
We also complement this result by showing that the dependence on $n$ in the bounds on $k$ is asymptotically the best possible.

    \begin{proof}[Proof of Theorem \ref{thmhyp1}] %!! explain the statement
        Let $G$ be as in the statement, an $r$-uniform hypergraph satisfying $\Delta_{\ell+1}(G) =1$ and $\Delta_i(G) = \Delta_i$ for $0 \leq i \leq \ell$. The vertex set of $G$ is $U$, of order at most $n$, and the vertex set of the complete $r$-graph $\Knr$ is $V$. We consider a coloring $c$ of $\Knr$, to which we impose different local or global restrictions.
        
        We set up some notation. In the hypergraph setting, it is more convenient for us to index the bad events in terms of edges. Firstly, fix a total ordering on the subsets of $U$ of order $r$, which induces an ordering on the edges of $G$. For $i \in \{0, 1,\dots, \ell \}$, a \emph{cherry} in $G$ of overlap $i$ is an ordered pair of edges $(e_1,\, e_2)$ satisfying $e_1 < e_2$ and $|e_1 \cap e_2| =i$. A vertex $u \in e_1 \cap e_2$ is called an \emph{apex vertex}. 

Let $\mathcal S$ be the set of all injections from $U$ to $V$.
 
    We will now define the canonical bad events in the uniform probability space on $\mathcal S$. Let $(e_1, e_2)$ be a cherry of overlap $i$ and $\tau: e_1 \cup e_2 \rightarrow V$ an injection satisfying $c(\tau (e_1)) = c(\tau(e_2))$. Notice that since $\tau$ is injective, the images $\tau(e_1) = \{\tau(v): v \in e_1 \}$ and $\tau(e_2)$ are edges of $\Knr$, and the colours $c(\tau (e_1))$, $c(\tau(e_2))$ are well-defined. Moreover, $|\tau(e_1) \cap \tau(e_2)| = i$ is preserved. We define the corresponding canonical bad event (of overlap $i$)

        $$ \badet = \left \{f\in{\mathcal S}: f(u) = \tau(u) \text{ for all } u \in e_1 \cup e_2 \right \}.$$
        We denote the set of all bad events of overlap $i$ by $\mathcal{B}_i$. Note that in all of the definitions above, we also allow $i=0$, corresponding to the case of disjoint edges.
        
We first consider a locally $k$-bounded coloring $c$ and prove
statement (i) - that $c$ is $G$-proper for
$k = O\left(\frac{n^{r-\ell}}{\Delta_1 \Delta_\ell} \right)$.
Let $\mathcal{B} = \mathcal{B}_1 \cup \dots \cup \mathcal{B}_{\ell}$ be the set
of bad events.  Lemma \ref{luszekely} grants that the graph on vertex set
$\mathcal{B}$ and edges between $\badet$ and $\badetp$ whenever the two events
$\mathcal S$-intersect is a negative dependency graph. By definition, this
occurs only when $e_1^\prime \cup e_2^\prime$ intersects $e_1 \cup e_2$, or
$\tau^\prime\left(e_1^\prime\right) \cup \tau^\prime\left(e_2^\prime\right)$ intersects $\tau(e_1) \cup \tau(e_2)$.
If the prior occurs, we call the events $G$-intersecting, and otherwise we call
them $K$-intersecting. In particular, if $u$ lies in
$\left(e_1^\prime \cup e_2^\prime\right) \cap \left(e_1 \cup e_2\right)$,
we say the two events $G$-intersect at $u$, and analogously we say the two
events $K$-intersect at $v$, if
$v$ lies in the intersection of $\tau^\prime\left(e_1^\prime\right) \cup \tau^\prime\left(e_2^\prime\right)$
and $\tau\left(e_1\right) \cup \tau\left(e_2\right)$.
        
        Each bad event $B = \badet$ satisfies $\pr{B} \leq \frac{1}{n(n-1)\dots (n-2r+\ell+1)} < \frac{1}{4}$. Equality holds when $B$ is of overlap $\ell$, and otherwise the probability is strictly smaller.  By Lemma \ref{asymmetric}, if every $B \in \mathcal{B}$ satisfies
        \begin{equation} \label{sumpb0}
            \sum_{\substack{
            B^\prime \in \mathcal{B} \\
            B^\prime \ {\mathcal S}\text{-intersects } B
            }} \pr{B^\prime} \leq \frac{1}{4},
        \end{equation}
        then with positive probability all events in $\mathcal{B}$ are avoided, i.e. $f(G)$ is a properly colored copy of $G$ in $\Knr$.
        
    Fix an event $B=\badet$ and vertices $u \in e_1 \cup e_2$, $v \in \tau(e_1) \cup \tau(e_2)$. We define $I_G(B, i, u)$ to be the number of events $B^\prime$ of overlap $i$ satisfying $u \in e_1^\prime \cup e_2 ^\prime$. Similarly, $I_K(B, i, v)$  is the number of events $B^\prime$ of overlap $i$ satisfying $v \in \tau^\prime(e_1^\prime) \cup \tau^\prime(e_2 ^\prime)$. Note that there are at most $2r-i\le2r$ choices of the vertex $u$ and similarly $2r-i\le2r$ choices of $v$.
    
    \textbf{Claim 1.} For all $i \in [\ell]$,
    \begin{equation*} \label{igbl}
        I_G\left(B, i, u\right) = O\left(n^{r}\Delta_1 \Delta_i k\right) \text{ \quad and \quad} 
        I_K(B, i, v) = O\left(n^{r}\Delta_1 \Delta_i k \right).
    \end{equation*}
    To see the first inequality, there are at most $\Delta_1$ ways of choosing an edge $e \in E(G)$ containing $u$, and at most $\binom{r}{i}$ ways to select the apex vertices from $e$. The vertices of $e$ can be mapped to any $r$-tuple of vertices in $V$, for which there are $n(n-1)\dots (n-r+1) \leq n^r$ choices.  From there, there are at most $\Delta_i$ ways to choose the other edge in $G$ forming a cherry of overlap $i$ with $e$, and whether $e = e_1 ^\prime$ or $e = e_2 ^\prime$ is determined by the ordering of edges. Without loss of generality, $e = e_1^\prime$.  Finally, as $\tau^\prime(e_1^\prime)$ is fixed, there are at most $k$ choices for $\tau^\prime(e_2^\prime)$ which form a monochromatic cherry in $\Knr$ with it. The number of orderings of $\tau^\prime(e_2^\prime)$ is suppressed into the constant, so altogether $I_G(B, i, u) = O\left(n^{r}\Delta_1 \Delta_i k\right)$.
    
    For the second inequality, there are $n$ ways to select the preimage $\tau^{-1}(v) \in U$, and $\Delta_1$ ways to extend it into an edge $e$ of $G$. Then we fix the image $\tau(w)$ for all vertices $w \in e \setminus\{\tau^{-1}(v)\}$, for which there are at most $(n-1)(n-2)\dots (n-r+1) \leq n^{r-1}$ ways. Now we are in the same position as above, so there are further $O(\Delta_i k)$ ways to completely determine $B^\prime$. Multiplying the different counts, we get $I_K(B, i, u) = O\left(n^{r}\Delta_1 \Delta_i k\right)$.
    
    \textbf{Claim 2.} For all $i \in [\ell]$,
        \begin{equation} \label{sumpb2}
                \sum_{\substack{
                B^\prime \in \mathcal{B}_i \\
                B^\prime \ \mathcal{S}\text{-intersects } B
                }} \pr{B^\prime} = O\left(n^{\ell-r} \Delta_1 \Delta_{\ell} k\right).
        \end{equation}
        Indeed, it is easy to see that $\Delta_{i-1}(G) \leq (n-i+1) \Delta_i(G)$ for every $i$ and every hypergraph $G$.
         
            Namely, for a vertex set $A \subset U$ with $|A|=i-1$ and any vertex $v \in U \setminus A$,  $d_G(A \cup \{v \}) \leq \Delta_i(G)$ holds by the definition of $\Delta_i$. Summing up over all the choices of $v$ gives $d_G(A) \leq (n-i+1)  \Delta_i(G)$.
        
        Iterating the inequality yields $\Delta_{i}(G) = O\left(n^{\ell-i} \Delta_{\ell}\right)$ for $i \leq \ell$. An event of overlap $i$ has probability exactly $\frac{1}{n(n-1)\dots (n-2r+i+1)} = O\left(n^{-2r+i}\right)$.
        Summing up $I_G(B,i,u)$ and $I_K(B,i,v)$ over all $i\in[\ell]$ and vertices $u\in U$ and $v\in V$, we get
        \begin{equation*} 
                \sum_{\substack{
                B^\prime \in \mathcal{B} \\
                B^\prime \ \mathcal{S}\text{-intersects } B
                }} \pr{B^\prime} = \sum_{i \in [\ell]} O \left( n^{r}k \Delta_1 n^{\ell-i} \Delta_{\ell}\right) \cdot O\left(n^{-2r+i}\right) = O\left(n^{\ell-r} \Delta_1 \Delta_{\ell} k\right) ,
        \end{equation*}
        as required.
        Therefore, setting $k = \frac{ c_1 n^{r-\ell }} {\Delta_1 \Delta_{\ell}}$ for $c_1>0$ sufficiently small, we get \eqref{sumpb0}, which completes the proof of the part (i).
        
        For the part (ii), suppose $c$ is $k$-globally bounded, where $k = \frac{c_2 n^{r-\ell}}{\Delta_1 \Delta_\ell}$ and $c_2 < c_1$ is a positive real we determine later. Note that since $c_2 < c_1$, the equation \eqref{sumpb2} still holds.
It is therefore enough to prove a bound analogous to the one in Claim 1 for the number of bad events $B'$ of overlap zero. To be precise, $f: U \rightarrow V$ is again a random injection, and the set of bad events is now $\mathcal{B} \cup \mathcal{B}_0$, where $\mathcal{B}_0$ contains the events of overlap zero. The negative dependency graph has edges exactly between the pairs of $\mathcal S$-intersecting events in $\mathcal{B} \cup \mathcal{B}_0$.
        
        Fix an event $B=\badet$, vertices $u \in e_1 \cup e_2$ and $v \in \tau(e_1) \cup \tau(e_2)$, and define $I_G(B, i, u)$ and $I_K(B, i, v)$ just like before.
        
        \textbf{Claim 3. } $I_G(B, 0, u) + I_K(B, 0, v) =O\left(n^r\Delta_0\Delta_1 k \right)$.\\
        Recall that $\Delta_0$ is just the number of edges of $G$. The count is exactly like before. For $I_G(B, 0, u)$, we select an edge $e$ of $G$ containing $u$ and an injection $\tau^\prime: e \rightarrow V$ in $O(\Delta_1 n^{r})$ ways, and then another edge and its image which matches $c(\tau^\prime(e))$ in $O(\Delta_0 k)$ ways. Multiplying gives $I_G(B, 0, u) = O\left(n^r\Delta_0\Delta_1 k\right).$
        The same holds for $I_K(B, 0, v)$ -- we can select an edge $e$ and an injection $\tau^\prime: e \rightarrow V$ so that $ v \in \tau^\prime(e)$ in  $O\left(\Delta_1 n^{r}\right)$ ways, and then we are in the same position as above. Summing up completes the proof of Claim 3.
        
        Introducing $\Delta_0 = O\left(n^{\ell}\Delta_{\ell}\right)$ and $\pr{B^\prime} = O\left(n^{-2r}\right)$ for events $B^\prime$ of overlap $0$ gives
        
        \begin{equation}   
                \sum_{\substack{
                B^\prime \in \mathcal{B} \cup \mathcal{B}_0 \\
                B^\prime \ {\mathcal S}\text{-intersects } B
                }} \pr{B^\prime} =  O\left(n^{\ell-r} \Delta_1 \Delta_{\ell} k \right) .  \label{sumsum}
        \end{equation}
        
        We set $k = O \left( \frac{ n^{r-\ell} }{\Delta_1 \Delta{_\ell}} \right)$.
        %\frac{c_g n^{r-\ell}}{\Delta_1 \Delta{\ell}}$ for a constant $c_g$ such that the right-hand side of \eqref{sumsum} is at most $\frac{1}{4}$.
        Lemma \ref{asymmetric} implies that then there exists an embedding $f$ avoiding all events in $\mathcal{B} \cup \mathcal{B}_0$, i.e. a rainbow embedding of $G$.
    \end{proof}
    
    We conclude the section with two constructions. Firstly, we show that the power $k = O\left(n^{r-\ell} \right)$ is the highest power of $n$ for which we can guarantee an embedding. Secondly, for $l=r-1$ and $\Delta_1(G)\Delta_{r-1}(G) = \Theta(n)$, we cannot hope for a stronger embedding result than Theorem \ref{thmhyp1}.
    
        For the first construction, let $D^{(r)}_\ell(m)$ be an $m$-vertex $r$-uniform hypergraph such that each set of $\ell+1$ vertices is contained in exactly one edge. Note that a recent result of Keevash on the existence of designs~\cite{keevash} grants existence of
    such hypergraphs whenever $m$ is sufficiently large and the parameters $r,\ell$ and $m$ satisfy all the necessary divisibility conditions.
    Also note that each $\ell$-subset of the vertices of $D^{(r)}_\ell(m)$ is contained in $m-\ell$ edges. In fact, we do not need a design, but only a hypergraph in which all the vertex subsets of order $\ell$ are contained in at least two edges, but each $(\ell+1)$-subset is contained in at most one edge. Such hypergraphs can be constructed probabilistically using standard nibble techniques.  
        \begin{PROP} \label{besthyp}
					 Let $r,\ell$ and $m$ be integers such that there exists a hypergraph
           $D^{(r)}_\ell(m)$. For any $n\ge m$ there exists a globally
					 $n^{r-\ell}$-bounded coloring of $\Knr$ which contains no  properly
            colored copy of $D^{(r)}_\ell(m)$.
        \end{PROP}
    
    \begin{proof}%[Proof of Theorem \ref{besthyp}]
        %Fix $r$ and $\ell \in [r-1]$.
        Let $V$ be the vertex set of $\Knr$ with a given ordering, and
        let $c$ be the following coloring of the edges of $\Knr$ using $\card{\binom{n}\ell}$ colors. For each $v_1 < v_2 < \dots < v_r \in V$, set 
        $$c(v_1,v_2,\dots,v_r) = \{v_1,v_2,\dots, v_\ell\}.$$
        That is, the edges of $\Knr$ are colored so that the color of each edge is uniquely determined by the first $\ell$ vertices. The coloring $c$ is globally $n^{r-\ell}$-bounded.
%        Now let $G$ be any $r$-uniform hypergraph satisfying $d_G(S) \geq 2$ for all vertex sets $S$ of order $|S| \leq \ell$. Note that the design $D^{(r)}(\ell)$ satisfies this.
Suppose there is a properly colored copy of $D^{(r)}_\ell(m)$ in $c$, and let $v_1, v_2,\dots, v_\ell$ be the minimal vertices in this copy. But then there are two edges $e_1$ and $e_2$ in this embedding of $G$ containing $v_1, v_2, \dots v_\ell$, so $c(e_1) = c(e_2) = \{v_1,v_2,\dots, v_r\}$, which is a contradiction.
    
    \end{proof}
        
        The second result is a hypergraph extension of the tree construction that the second two authors used in \cite{volec}.
        \begin{PROP} For any $r \geq 2$, there is an $r$-uniform hypergraph $G$ on $n$ vertices satisfying $\Delta_1(G)\Delta_{r-1}(G) = \Theta(n)$, and a globally $O(1)$-bounded coloring of $\Knr$ which is not $G$-proper.
        \label{thm:tree}
        \end{PROP}
            \begin{proof}
            Let $n_1$ be a natural number and $n = 1 + n_1+ \binom{n_1}{r-1}n_1 = \Theta \left(n_1^{r} \right)$.  Let $G$ have vertex set $V = \{v\} \cup L_1 \cup L_2$. Here $|L_1| = n_1$ are vertices in the first level, which means that each $(r-1)$-tuple in $L_1$ forms an edge with the \emph{root} $v$. Furthermore, each $(r-1)$-tuple $S \subset L_1$ has its own $n_1$ \emph{children} belonging to $L_2$. In other words, for each $u \in L_2$, there is a unique $(r-1)$-tuple $S \subset L_1$ such that $S \cup \{ u\}$ is an edge. 
            Indeed, $G$ has the $1$-degree $\Delta_1(G) = \Theta\left(n_1^{r-1} \right)$ attained by $v$ and the vertices in the first level. Moreover, $\Delta_{r-1}(G) \leq n_1$ by looking at any $r-1$ vertices in $L_1 \cup \{ v \}$.
            
            Next, we give the promised coloring $c$ of $\Knr$. Let $n$ be partitioned into sets $S_1, S_2, \dots, S_{\frac{n}{r+1}}$ of order $r+1$. An edge $\{u_1, u_2, \dots u_r \}$ satisfying $u_1 \in S_{i_1}$, $u_2 \in S_{i_2}$, up to $u_r \in S_{i_r}$ gets the color $\{i_1, i_2, \dots i_r \}$ (viewed as a multiset). The coloring is globally $(r+1)^r$-bounded. Suppose that the bijection $f: V \to [n]$ induces a properly colored copy of $G$, and that, without loss of generality, the image of $v$ lies in $S_1$. If the other $r$ members of $S_1$ lie in $f(L_1)$, then they span $r$ edges of color $\{1, 1, \dots 1 \}$. Otherwise, let $w \in L_2$ satisfy $f(w) \in S_1$. There is an $(r-1)$-tuple $S \subset L_1$ which forms an edge in $G$ with $w$, and an edge with $v$. The edges $f(S \cup \{ v\}) $ and $f(S \cup \{ w\}) $ have the same color in the embedding given by $f$.
            \end{proof}
    
    \section{Concluding remarks}
In this paper we studied the problem of finding a rainbow and properly colored copy of a graph/hypergraph $G$ with some degree restrictions
in a bounded $k$-coloring of the complete multipartite graph/hypergraph.
We obtained upper bounds on $k$ in terms of maximum degree/$\ell$-degree of $G$
that guarantees locally and globally $k$-bounded colorings to be $G$-proper and $G$-rainbow, respectively. Moreover, for multipartite graphs, the dependence of $k$ on other parameters in our bounds is the best possible up to a constant factor. However, there are several natural questions which remain
open. Here we mention two of them that we find the most interesting.  We state them only for properly colored copies of hypergraphs in locally bounded colorings, but the analogues for globally bounded colorings seeking rainbow copies are just as interesting.

The first question asks for the correct asymptotics of $k$ that guarantee a properly colored copy
of a tight Hamilton cycle in locally $k$-bounded colorings of $\Knr$.
What is the largest possible $k$, such that any locally $k$-bounded coloring of $\Knr$ is $C_n^{r}(r-1)$-proper?
In particular, for $r=3$, is there an $\eps >0$ such that any locally $O(n^{1+\eps})$-bounded coloring of $K_n^{(3)}$ contains a tight Hamilton cycle?

We have shown that the dependence on $n$ in Theorem \ref{thmhyp1} is the best
possible. However, apart from the case $l = r-1$, we do not know if the dependence on the maximum degrees $\Delta_1,\Delta_2,\dots,\Delta_\ell$ of $G$ is the correct one or not.
The first unresolved case are the $3$-uniform linear hypergraphs, i.e.~hypergraphs $G$ satisfying $\Delta_2(G)=1$. Are there 3-uniform $n$-vertex \emph{linear} hypergraphs $G$ and $O \left(\frac{n^2}{\Delta_1(G)^2} \right)$-bounded colorings which are not $G$-proper?

\vspace{0.3cm}

{\bf Acknowledgements.}\,  The authors would like to thank Oriol Serra for
fruitful discussions on the topics related to this project and pointing
out the references~\cite{cano} and~\cite{lumohrszek} to us. We are grateful to the anonymous referees for their valuable comments, which improved the presentation of the results.

\pagebreak
\appendix
\section{Multidimensional Lu-Sz\'ekely}
    
We now present a proof of Theorem \ref{luszekely}. Note that the case $m=1$ is
a result of Lu and Sz\'ekely from~\cite{lu}. As in~\cite{lu}, we actually show
that the following slightly stronger choice of a graph gives a negative
dependency graph.

We say that two canonical events  $\Omega(T_1, U_1,
\tau_1)$ and $\Omega(T_2, U_2, \tau_2)$  in the probability space $\Omega$
\emph{conflict} if
    $$ \exists x \in T_1 \cap T_2: \tau_1(x) \neq \tau_2(x) \quad \text{or} \quad \exists y \in U_1 \cap U_2 : \tau_1^{-1}(y) \neq \tau_2^{-1}(y).$$
    Clearly two conflicting events are disjoint, and therefore negatively correlated. We now show that just connecting the conflicting events suffices for a negative dependency graph, even if we require the injections to respect a given partition of $X$ and $Y$. 
 
    \begin{THM} \label{luszekely1}
     Let $X = X_1 \times \dots \times X_m$ and $Y = Y_1 \times \dots Y_m$, where the parts $X_k$ and $Y_k$ satisfy $|X_k| \leq|Y_k| $ for all $k \in [m]$. Consider the probability space $\Omega$ generated by picking a uniformly random injection $\sigma: X \rightarrow Y$ satisfying $\sigma(X_k) \subset Y_k$ for all $k$. Denote the set of such injections by $\mathcal{S}$.
    
     Let $B_1, B_2, \dots , B_N$ be canonical events in $\Omega$, and define the graph $D^\prime$ on $[N]$ by
     $$E(D^\prime) = \{ij : B_i \text{ and }B_j \text{ conflict} \}.$$
     Then $D^\prime$ is a negative dependency graph for the events $B_1, \dots B_N$. 
        \end{THM}
        Since the dependency graph $D^\prime$ is a subgraph of the graph $D$ in Theorem \ref{luszekely} with the edges between pairs of $\mathcal S$-intersecting events, it follows that $D$ is also a negative dependency graph for the same set of events.
    \begin{proof}
        Our proof follows the outline of \cite[Theorem 1]{lu}, but there are several claims we need to verify in the multidimensional setting. 
        
         A \emph{matching} between $X$ and $Y$ is a triple $(T, U, \tau)$, where $\tau$ is a part-respecting bijection from $T \subset X$ to $U \subset Y$, that is $\tau(X_k \cap T) \subset Y_k$ for all $k$. All the functions we consider will be part-respecting. Fix an event $B_i = \Omega(T, U, \tau)$ for a matching $(T, U, \tau)$, and a set of indices $J =J(i) \subset \{j \in [N]: ij \notin E(D^\prime) \}$. We are to show the inequality $\pr{B_i \mid \wedge_{j \in J} \overline{B_j}} \leq \pr{B_i}$, which is equivalent to
        \begin{equation}
            \label{negdep}
            \pr{\bigwedge_{j \in J} \overline{B_j} \mid B_i } \leq \pr{\bigwedge_{j \in J} \overline{B_j}}.
        \end{equation} %!!
        
        Here we assume that $\pr{\bigwedge_{j \in J} \overline{B_j}}>0$, otherwise there is nothing to prove. The inequality follows immediately from the following claim.
        
        \textbf{Claim.} For any canonical event $B^\prime = \Omega(T, U^\prime, \tau^\prime)$,
        \begin{equation} \label{symmetry}
            \pr{ \left(\bigwedge_{j \in J} \overline{B_j} \right) \wedge B_i} \leq 
            \pr{ \left(\bigwedge_{j \in J} \overline{B_j} \right) \wedge B^\prime}.
        \end{equation}
        
        Intuitively, the claim says that upon $\wedge_{j \in J} \overline{B_j}$, there is no mapping of $T$ that is less likely than $ T \stackrel{\tau}{\rightarrow} U$.
        
        \textbf{Proof of Claim.} Fix a canonical event $B^\prime = \Omega(T, U^\prime, \tau^\prime)$. Let $J^\prime = J^\prime(J(i), B^\prime)$ be the set of indices $j \in J$ so that $B_j$ does not conflict $B^\prime$. If $j \in J \setminus J^\prime$, then $B_j$ conflicts $B^\prime$, that is $B^\prime$ implies $\overline{B_j}$. Therefore
        \begin{align*}
            \overline{B_j} \wedge B^\prime = B^\prime, \text{\quad so \quad }% \\
            \left(\bigwedge_{j \in J} \overline{B_j} \right) \wedge B^\prime = \left(\bigwedge_{j \in J^\prime} \overline{B_j} \right) \wedge B^\prime .
        \end{align*}
        
        The idea is that the functions $\sigma \in \wedge_{j \in J^\prime} \overline{B_j}$ are equally likely to map $T$ via $\tau$ to $U$ as they are to map it via $\tau^\prime$ to $U^\prime$. Formally, we construct an automorphism of the probability space $\Omega$ which fixes each of the events $B_j = (T_j, U_j, \tau_j)$ for $j \in J^\prime$, but maps $B^\prime$ to $B_i$.
        
        To do this, we set up some notation. Let $\rho$ be a part-respecting permutation of $Y$, i.e. a bijection $Y \rightarrow Y$ satisfying $\rho (Y_k) \subset Y_k$ for all $k \in [m]$. The permutation $\rho$ determines an action on the matchings
        $$\pi_\rho (\sigma) = \rho \circ \sigma \text{ \quad for all } (P, Q, \sigma).$$
        Clearly $\pi_\rho$ takes any matching $(P, Q, \sigma)$ to a matching $(P, \rho(Q), \rho \sigma)$ with the same domain. Furthermore, it preserves the uniform measure on $\Omega$ (for this it is crucial to notice that $\rho$ preserves parts of $Y$). Finally, the action of  $\pi_\rho$ on the canonical events in $\Omega$ is described by
        $$\pi_\rho(\Omega(P, Q, \sigma))= \Omega(P, \rho(Q), \rho \sigma).$$
        
        Now we are ready to show that
        \begin{equation}
        \label{eq10}
            \pr{ \left(\bigwedge_{j \in J^\prime} \overline{B_j} \right) \wedge B_i} = \pr{ \left(\bigwedge_{j \in J^\prime} \overline{B_j} \right) \wedge B^\prime}.
        \end{equation}

        Let $\rho: Y \rightarrow Y$ be a part-respecting bijection defined by
        $$\rho(y) = y \text{ for any } y \in Y \setminus U^\prime \text{\quad and \quad} \rho(y) = \tau (\tau^{\prime{-1}}(y)) \text{ for } y \in U^\prime.$$
        We first check that $\rho$ fixes points of $U_j$ for $j \in J^\prime$. This is clear for $y \in U_j \setminus U^\prime$. For $y \in U_j \cap U^\prime$, there is an $x \in T$ with $\tau^\prime (x) = y$. Since the events $B_j$ with $j \in J^\prime$ conflict neither $B_i$ nor $B^\prime$, this $x$ satisfies $\tau (x)= \tau_j(x) = \tau(x)^\prime =y$, so indeed $\rho(y) = y$, and therefore  $\pi_\rho(T_j, U_j, \tau_j) = (T_j, U_j, \tau_j)$. Furthermore, $\rho(\tau^\prime(x)) = \tau(x) $ for every $x \in T$, so $\pi_\rho(T, U^\prime, \tau^\prime) = (T, U, \tau)$. 
        
        This proves Equation \eqref{eq10}, since the two events correspond to each other under the automorphism $\pi_\rho$. Hence
        \begin{align*}
            \pr{ \left(\wedge_{j \in J} \overline{B_j} \right) \wedge B_i} &\leq \pr{ \left(\wedge_{j \in J^\prime} \overline{B_j} \right) \wedge B_i}\\
            &= \pr{ \left(\wedge_{j \in J^\prime} \overline{B_j} \right) \wedge B^\prime} \\
            &= \pr{ \left(\wedge_{j \in J} \overline{B_j} \right) \wedge B^\prime}.
        \end{align*}
        This proves the claim. Keeping $T$ fixed, the events $B^\prime = \Omega(T, U^\prime, \tau^\prime)$ across all the matchings $(T, U^\prime, \tau^\prime)$ partition the probability space $\Omega$, so summing up equation \eqref{eq10} over all such $B^\prime$ gives
        \begin{align*}
            \pr{ \wedge_{j \in J^\prime} \overline{B_j} } &= \sum_{B^\prime} \pr{ \left(\wedge_{j \in J^\prime} \overline{B_j} \right) \wedge B^\prime} \\
            & \geq \sum_{B^\prime} \pr{ \left(\wedge_{j \in J^\prime} \overline{B_j} \right) \wedge B_i}\\
           &= \sum_{B^\prime} \pr{ \left(\wedge_{j \in J^\prime} \overline{B_j} \right) \mid B_i} \pr{B_i}\\
            &= \sum_{B^\prime}\pr{ \left(\wedge_{j \in J^\prime} \overline{B_j} \right) \mid B_i} \pr{B^\prime}\\
            &= \pr{ \left(\wedge_{j \in J^\prime} \overline{B_j} \right) \mid B_i},
        \end{align*}
        where in the fourth line we used the fact that $\pr{B_i} = \pr{B^\prime}$ by the uniformity of our probability space $\Omega$.
    \end{proof}        
   Since any two conflicting events in the space $\Omega$ are $\mathcal S$-intersecting, Theorem~\ref{luszekely} immediately follows.

\end{document}